\documentclass[11p,leqno]{amsart}
\usepackage{amsmath}
\usepackage{amsfonts}
\usepackage{amssymb}
\usepackage{graphicx}
\usepackage{color}
\usepackage{amsfonts,mathtools,mathrsfs}
\usepackage{amscd}
\usepackage{amsmath,amsthm}
\newtheorem{theorem}{Theorem}[section]\newtheorem{thm}[theorem]{Theorem}
\newtheorem*{theorem*}{Theorem}
\newtheorem{lemma}{Lemma}[section]
\newtheorem{corollary}[theorem]{Corollary}

\newtheorem{remark}[theorem]{Remark}

\def \b {\beta}

\def\a{\alpha}

\def\ol{\overline}
\def\ul{\underline}
\def\g{\gamma}
\def\b{\beta}

\def\p{\partial}
\def\l{\lambda}

\def\R{\mathbb R}

\def\tu{\tilde{u}}
\def\vp{\varphi}

\numberwithin{equation}{section}

\begin{document}
\title[]{\bf Nonparametric Hypersurfaces Moving by powers of Gauss Curvature}

\author{Xiaolong Li}
\address{Department of Mathematics, University of California, San Diego, La Jolla, CA 92093, USA}
\email{xil117@ucsd.edu}
%
%
\author{Kui Wang}
%
%
\address{School of Mathematic Sciences, Soochow University, Suzhou, 215006, China}
\email{kuiwang@suda.edu.cn}

\thanks{ The research of the first author is partially supported by an Inamori fellowship and NSF grant DMS-1401500.  }

\maketitle

\begin{abstract}
We study asymptotic behavior of nonparametric hypersurfaces moving by $\a $ powers of Gauss curvature with $\a >1/n$. 
Our work generalizes the results of V. Oliker \cite{O1} for $\a =1$. 
\end{abstract}
\section{Introduction}

Let $\Omega$ be a bounded strictly convex domain in $\R^n$, $n\geq 2$, with smooth boundary $\p \Omega$. 
We consider a solution of the following initial boundary problem
\begin{align}\label{PGCF}
& u_t =\frac{[\det(u_{ij})]^{\a}}{(1+|\nabla u|^2)^{\a \b}} \text{  in  } \Omega\times(0,\infty), \nonumber \\
& u(x,t) =0 \text{  in  } \p \Omega \times (0,\infty), \\
& u(x,t) \text{  is strictly convex for each  }  t \geq 0, \nonumber 
\end{align}
where $\a >1/n$ and $\b \geq 0$ are constants and 
$$u_t :=\frac{\p u}{\p t}, \text{  } u_{ij}:=\frac{\p^2 u}{\p x_i \p x_j}, 
\text{  } \nabla u :=\left(\frac{\p u}{\p x_1}, \cdots , \frac{\p u}{\p x_n}\right).$$

Equation \eqref{PGCF} describes the graphs 
$(x, u(x,t)), (x,t) \in \ol{\Omega} \times [0,\infty) $
evolving in $\R^{n+1}$ with relative boundaries $\left.(x,u(x,t)) \right|_{\p \Omega}$ remain fixed.  
When $\b=\frac{n+2-\frac{1}{\a}}{2}$, the normal speed of the point $(x,u(x,t))$ 
is equal to $\a$ powers of the Guass curvature of the graph.  
Such parabolic Monge-Amp\`ere equations have been studied by many authors in recent years.  
See, for instance, 
\cite{HL06}\cite{DS12}. 
On the other hand, in the parametric setting, 
flow by Gauss curvature or its powers have received considerable interests, 
see \cite{Chou85}\cite{Chow85}\cite{Chow91}\cite{Andrews99}
\cite{Andrews00}\cite{GN15}\cite{AGN}
and the references therein. 



V. Oliker considered \eqref{PGCF} with $\a =1$ in \cite{O1}. 
He analyzed the asymptotic behavior of smooth convex solutions of \eqref{PGCF}.
It turned out that solutions with different $\b$ all have the same asymptotic behavior. 
Moreover, if $\Omega$ is centrally symmetric or rotationally symmetric, 
then the solution $u(x,t)$ asymptotically becomes centrally symmetric or rotational symmetric, 
regardless of its initial shape.     

The goal of this paper is to generalize V. Oliker's results in \cite{O1} to any power $\a >1/n$. 
We investigate the asymptotic behavior of a smooth convex solution of \eqref{PGCF} and
show that, by comparing with self-similar solutions of \eqref{PGCF} with $\b=0$, 
the solution $u(x,t)$ asymptotically converges to the solution of the following nonlinear elliptic problem: 
 \begin{align*}
 & [\det(\psi_{ij})]^{\a} =  \frac{1}{1-n \a} \psi \text{  in  } \Omega, \text{  } \psi=0 \text{  on  } \p \Omega, \\
  & \psi \text{  is stictly convex and  } \psi <0 \text{  in  } \Omega. \nonumber
 \end{align*}
Furthermore, our estimate implies geometric properties of the flow by $\a$ powers of the Gauss curvature. 
For instance, the asymptotic behavior of $u(x,t)$ reflects the symmetries of $\Omega$. 
More precisely, if $\Omega$ is centrally or rotationally symmetric, 
then the solution $u(x,t)$ asymptotically becomes centrally or rotational symmetric, 
regardless of its initial shape, and we also give sharp estimates on the rate of this process.

Throughout out the paper, we 
denote by $M$ the Monge-Amp\`ere operator 
$M(u):=\det (u_{ij})$ and $M ^{\a}(u):=[\det(u_{ij})]^{\a}$. 

\section{Main Results} 

Consider the following initial boundary problem:
\begin{align} \label{MA}
& u_t =M ^{\a}(u)\text{  in  } \Omega \times(0,\infty), \nonumber \\
& u(x,t) =0 \text{  in  } \p \Omega \times (0,\infty), \\
& u(x,t) \text{  is strictly convex for each  }  t \geq 0. \nonumber 
\end{align}
 
We seek for self-similar solutions of \eqref{MA} of the form 
\begin{equation}\label{1.1}
u(x,t)=\vp(t) \psi(x),
\end{equation}
where $\vp(t) \in C^{\infty}([0,\infty))$ and $\psi(x) =C^{\infty}(\Omega) \cap  C^{0,1}(\ol{\Omega})$. 
By convexity of $u(x,0)=\vp(0)\psi(x)$, we have 
either $\vp(0) <0$ and $\psi(x) >0$ in $\Omega$ and concave
or $\vp(0) >0$ and $\psi(x) <0$ in $\Omega$ and convex. 
Since both cases are equivalent for our purpose, 
we always deal with the latter one. 
Substituting \eqref{1.1} into \eqref{MA} yields
$$\frac{\vp(t)}{\vp^{n \a}} = \frac{M^{\a}(\psi)}{\psi} =\l =\text{constant}. $$  
Noting that $\psi(x) <0$ and convex in $\Omega$, we get $\l \leq 0$ and 
\begin{equation}\label{1.2}
\vp(t) =\left(\vp(0)^{1-n \a} -(n\a -1) \l t\right)^{\frac{1}{1-n \a}},
\end{equation}
\begin{equation}\label{1.3}
M(\psi)=(\l \psi)^{\frac{1}{\a}} \text{  in  } \Omega \text{  and  } \psi =0 \text{  on  } \p \Omega.
\end{equation}
An easy argument shows that $\l =0$ implies $u(x,t) \equiv 0$. 
Thus we only consider the case $\l <0$. 
By scaling, it suffices to consider one negative value of $\l$ 
and thus we fix $\l =\frac{1}{1-n \a}<0$ 
for convenience. 
The following result establishes the existence of self-similar solutions to \eqref{MA}.

\begin{thm}\label{Thm A}
Let $\Omega$ be a bounded strictly convex domain with smooth boundary $\p \Omega$. 
Then problem \eqref{MA} 
admits a self-similar solution in $\ol{\Omega}\times (0,\infty) $ given by
\begin{equation}
u(x,t)=(1+t)^{\frac{1}{1-n\a}} \psi(x),
\end{equation} 
where $\psi$ is the unique solution in $C^{\infty}(\Omega) \cap C^{0,1}(\ol{\Omega})$ of the equation
\begin{align}\label{EMA}
& M(\psi) =  \left(\frac{-\psi}{|1-n\a|} \right)^{\frac{1}{\a}} \text{  in  } \Omega, \text{  } \psi=0 \text{  on  } \p \Omega, \\
& \psi \text{  is stictly convex and  } \psi <0 \text{  in  } \Omega, \nonumber
\end{align}
and $\sup_{\Omega} |\psi(x)|$ admits an estimate depending only on $n$, $\a$ and the domain $\Omega$. 
Furthermore, if $\tu(x,t) =\vp(t) \tilde{\psi}(x) $ is an arbitrary self-similar solution of \eqref{MA}, 
then there exists a unique $c>0$ such that $\tilde{\psi}(x) =c\psi(x)$ and 
\begin{equation}\label{1.4}
\tu(x,t)=u(x,t) \left\{\frac{1+t}{[c\vp(0)]^{1-n\a}+t}\right\}^{\frac{1}{n \a -1}}. 
\end{equation}
\end{thm}

The main theorem concerning the asymptotic behavior of the solution is the following: 
\begin{thm} \label{Thm B}
Let $u(x,t) \in C^2(\ol{\Omega}\times (0,\infty))$ be a solution of the problem
\begin{align} \label{PGCF2}
& u_t =\frac{M^{\a}(u)}{(1+|\nabla u|^2)^{\a \b}} \text{  in  } \Omega\times(0,\infty), \nonumber \\
& u(x,t) =0 \text{  in  } \p \Omega \times (0,\infty), \\
& u(x,t) \text{  is strictly convex for each  }  t \geq 0, \nonumber 
\end{align}
where $\a > 1/n$ and $\b \geq 0$ are constants.  
If $\b =0$, then there exists positive constant $C_1$ depending only on dimension $n$, $\a $, $\Omega$ and $u(x,0)$, such that 
for all $t\geq 0$, 
\begin{equation}\label{2.1}
\sup_{\Omega} \left|(1+t)^{\frac{1}{n \a -1}} u(x,t) -\psi(x) \right|\leq \frac{C_1}{1+t},
\end{equation}

If $\b >0$, then 
\begin{equation}\label{2.2}
\left[\frac{C_2}{1+t} +G^{\frac{1}{1-n \a}}-1\right] \psi 
\leq  (1+t)^{\frac{1}{n \a -1}} u(x,t) -\psi(x) 
\leq \frac{-C_3 \psi}{1+t},
\end{equation}
where $C_2$ and $C_3$ are positive constants depending only on dimension $n$, $\a $, $\Omega$,  $u(x,0)$ and 
\begin{equation*}
G=\inf_{\Omega} \left(1+|\nabla u(x,0)|^2\right)^{-\a \b}.
\end{equation*}
Moreover, 
\begin{equation}\label{2.3}
\lim_{t\to \infty} (1+t)^{\frac{1}{n \a -1}} u(x,t) =\psi(x) \text{  uniformly on  } \ol{\Omega}.   
\end{equation}
\end{thm}

We have gradient estimates for solutions of \eqref{PGCF2}. 
\begin{corollary}\label{Cor 1}
Suppose the same conditions as in Theorem \ref{Thm B} holds. Then for all $t\geq 0$, 
\begin{equation*}
\sup_{\Omega}|\nabla u(x,t)| \leq G^{\frac{1}{1-n\a}} \sup_{\p \Omega} \psi_{\nu}(x) (C_4+t)^{\frac{1}{1-n\a}}
\end{equation*}
where $\psi_{\nu}$ is the derivative in the direction of the outward unit normal to $\p \Omega$, 
and $C_4$ depends only on $u(x,0)$. 
\end{corollary}
An interesting geometric consequence of Theorem \ref{Thm B} is the following: 
\begin{thm}\label{Thm C}
If $\Omega $ is a ball in $\R^n$ and $u(x,t) \in C^2(\ol{\Omega}\times (0,\infty))$ is a solution of \eqref{PGCF2}. Then 
$$(1+t)^{\frac{1}{n \a -1}} u(x,t) \to \psi(|x|) \text{  uniformly on  } \ol{\Omega} \text{  as  } t\to \infty. $$
\end{thm}
This theorem implies that, 
$u(x,t)$ asymptotically becomes radially symmetric regardless of the initial shape. 
More generally, if $\Omega$ is centrally symmetric, then 
$$(1+t)^{\frac{1}{n \a -1}} u(x,t) \to \psi(x) \text{  uniformly on  } \ol{\Omega} \text{  as  } t\to \infty , $$
where $\psi(x)=\psi(-x)$. 
The proof of Theorem \ref{Thm C} is the same as in \cite[Section 6]{O1} and we omit it here.

\section{Proof of Theorem \ref{Thm A}}
\begin{proof} 
It was shown in \cite[Corollary 4.2, in which (2) should read as (1.2)]{Chou90} that for any $\a >1/n$, 
problem \eqref{EMA} admits a unique strictly convex solution $\psi$ in $C^{\infty}(\Omega) \cap C^{0,1}(\ol{\Omega})$.  
Direct calculation shows $u(x,t)=(1+t)^{\frac{1}{1-n\a}} \psi(x)$ solves \eqref{MA} with initial data $u_0(x)=\psi(x)$.
Next we prove $\sup_{\Omega} |\psi(x)|$ depends only on $n, \a $ and $\Omega$. 
Since $\psi$ is strictly convex and vanishes on $\p \Omega$, there exists a point $\bar{x} \in \Omega$ such that 
$\sup_{\Omega} |\psi| =|\psi(\bar{x})|$. 
Consider a cone $K$ generated by the linear segments joining the vertex $(\bar{x}, \psi(\bar{x}))$ with points on $\p \Omega$. 
Denote $\theta(x), x \in \ol{\Omega}$, the function whose graph is $K$. 
Obviously, $\theta \geq \psi$ in $\Omega$ and $\theta = \psi=0$ on $\p \Omega$. 
Then by \cite[Lemma 1.4.1]{GC1}
$M\theta (\Omega) \leq M\psi(\Omega)$, 
where $Mu$ denotes the Monge-Ampere measure associated with the function $u$(see \cite[Theorem 1.1.13]{GC1}).  
Since $\psi$ is $C^\infty$ and convex on $\Omega$, 
\begin{equation}
M\psi(\Omega)=\int_{\Omega} M(\psi)  =
\int_{\Omega} (\l \psi)^{\frac{1}{\a}} 
\leq |\l |^{\frac{1}{\a}} |\psi(\bar{x})|^{\frac{1}{\a}} |\Omega|.
\end{equation}
On the other hand, the Aleksandrov-Bakelman-Pucci maximum principle (see, for instance, \cite[Theorem 1.4.5]{GC1}) says 
$M\theta(\Omega) \geq \omega_n |\psi(\bar{x})|^n (\text{diam} \Omega)^{-n}$, 
where $\omega_n$ is the volume of the unit ball in $\R^n$. 
Thus 
\begin{equation}
\sup_{\Omega} |\psi(x)| = |\psi(\bar{x})| \leq \left(\frac{|\l|^{\frac{1}{\a}}|\Omega| (\text{diam} \Omega)^n}{\omega_n }\right)^{\frac{\a }{n \a -1}}.
\end{equation}
Finally, the proof of \eqref{1.4} parallels that in \cite[Section 4.3]{O1}.
\end{proof}

\begin{remark}
One can prove Theorem \ref{Thm A} without using the existence results from \cite{Chou90}. 
V. Oliker \cite{O1} proved that 
\eqref{EMA} has a unique solution in $C^{\infty}(\Omega) \cap C^{0,1}(\ol{\Omega})$ when $\a =1$. 
A careful examination of his proof shows it works indeed for all $\a >1/n$. 
\end{remark}
\begin{remark}
When $\a =1/n$, it was shown by P. L. Lions\cite{L1} that 
\begin{equation}\label{2.5}
M(\psi) = \mu (-\psi)^n \text{  in  } \Omega, \text{    }\psi =0 \text{  on  } \p \Omega
\end{equation}
admits a unique solution pair $(\mu, \psi)$ in the sense that 
if $(\nu, \phi)$, where $\nu$ is positive and $\phi$ is convex, solves \eqref{2.5}, 
then we must have $\mu =\nu$ and $\phi$ is a constant multiple of $\psi$.  
The number $\mu$ is called the first (in fact the only) eigenvalue of the Monge-Amp\`ere operator $M$, 
and the corresponding (normalized) eigenfunction is in $C^{\infty}(\Omega) \cap C^{1,1}(\ol{\Omega})$.  
The asymptotic behavior for $\a =1/n$ remains interesting and open.       
\end{remark}
\begin{remark}
When $0<\a <1/n$, K. Tso\cite[Theorem E]{Chou90} showed that 
\eqref{EMA} admits a convex solution in $C^{\infty}(\Omega) \cap C^{0,1}(\ol{\Omega})$. 
The uniqueness, however, is not known. 
In this case, the reader will see easily from the comparison with self-similar supersolutions in Section 4 
that smooth convex solutions of \eqref{PGCF2} must vanish at finite time. 
\end{remark}

\section{Proof of Theorem \ref{Thm B}}
In this section, 
we determine the asymptotic behavior of $u$ by comparing with self-similar solutions of \eqref{MA}.
A direct generalization of the proof given by V. Oliker in \cite{O1} works for $\a \geq 2/n$. 
New estimates are introduced in the following lemma to take care of the case $1/n < \a < 2/n$. 
\begin{lemma}\label{Lemma F}
Let $F:(0,S)\times [0,\infty) \to (0,\infty), S< \infty$ be defined by 
\begin{equation}
F(s,t) =\left(\frac{1+t}{s+t}\right)^{\frac{1}{n\a-1}} \equiv \left(1+\frac{1-s}{s+t}\right)^{\frac{1}{n\a-1}}.
\end{equation}
Then we have for all $t\geq 0$,
\begin{align}
F(s,t) & 
\leq 1+\frac{1}{n\a-1}\frac{1-s}{s(1+t)}, 
&\text{  if   } s\leq 1, \a \geq  2/n; \label{F1} \\ 
F(s,t) &\leq   1+\frac{1}{n\a-1} \left(\frac{1}{s}\right)^{\frac{1}{n \a -1}}\frac{1-s}{1+t},  
&\mbox{  if  } s \leq 1, \a \leq  2/n; \label{F2} \\
F(s,t) & 
\geq 1-\frac{s-1}{1+t}, 
&  \mbox{  if  } s\geq 1, \a \geq 2/n; \label{F3}\\
F(s,t) & 
\geq 1-{\frac{1}{n\a-1}} \frac{s-1}{1+t} 
&\mbox{  if  } s\geq 1, \a \leq  2/n. \label{F4} 
\end{align}
\end{lemma}
\begin{proof}
This lemma follows from elementary calculus. 
When $\a \geq 2/n$,  $\g :=\frac{1}{n\a-1} \leq 1$. 
Then \eqref{F1} follows from $(1+x)^{\g} \leq 1+\g x$ for all $x \geq 0$
and \eqref{F3} follows from $x^\g \geq  x $ for all $0\leq x\leq 1$.  
When $\a \leq 2/n$,  $\g :=\frac{1}{n\a-1} \geq 1$. 
Now \eqref{F2} is a consequence of $(1+x)^{\g} \leq 1+\g(1+a)^{\g -1} x$ for all $ 0\leq x \leq a$
and \eqref{F4} is a consequence of $(1+x)^{\g} \geq 1+\g x$ for all $ -1 < x \leq 0$. 
\end{proof}

\begin{proof}[Proof of Theorem \ref{Thm B}]
First of all, a uniform estimate of $|\nabla u(x,t)|$ is obtained similarly as in \cite{O1}. For any $t\geq 0$, 
\begin{equation}\label{a.1}
\sup_{\ol{\Omega}} |\nabla u(x,t)| \leq \sup_{\p \Omega} |\nabla u(x,t)| 
=\sup_{\p \Omega} |u_{\nu}(x,t)| \leq \sup_{\p \Omega}|u_{\nu}(x,0)|.
\end{equation}
Self-similar subsolution and supersolution are then constructed as follows:  
Let $$G=\inf_{\Omega} \left(1+|\nabla u(x,0)|^2\right)^{-\a \b}.$$
Clearly we have $0<G\leq 1$. It follows from \eqref{a.1} that
\begin{equation}
GM^{\a}(u) \leq \left(1+|\nabla u(x,t)|^2\right)^{-\a \b} M^{\a}(u) =u_t \text{  in  } \Omega \times (0,\infty). 
\end{equation}
Put 
$\underline{u}(x,t)=G^{\frac{1}{1-n \a}} \underline{\vp}(t)\psi(x)$ and
$\ol{u}(x,t)=\ol{\vp}(t)\psi(x)$, 
where $\psi$ is the solution of \eqref{EMA} and 
\begin{align*}
\underline{\vp}(t)=\left(\underline{\vp}(0)^{1-n \a} +t\right)^{\frac{1}{1-n \a}},\\
\ol{\vp}(t)=\left(\ol{\vp}(0)^{1-n \a} +t\right)^{\frac{1}{1-n \a}}.
\end{align*}
Then $\underline{u}$ and $\ol{u} $ satisfy
$\underline{u}_t =GM^{\a}(\underline{u}) $ 
and  $\ol{u}_t =M^{\a}(\ol{u}) \text{  in  } \Omega \times (0,\infty) $, respectively. 
Finally we define $\tu(x,t)=\underline{u}(x,t)-u(x,t)$ and it satisfies
\begin{equation}
\tu_t =GM^{\a}(\underline{u}) - \left(1+|\nabla u(x,t)|^2\right)^{-\a \b} M^{\a}(u) 
\leq GM^{\a}(\underline{u})-GM^{\a}(u) \text{  in  } \Omega \times (0,\infty).
\end{equation}
Observe that the operator $L(\tu)=M^{\a}(\underline{u}) -M^{\a}(u)$ is elliptic since 
\begin{equation*}
L(\tu) = \sum_{ij} \left(\int_{0}^{1}\a \det(u_{\tau i j})^{\a -1} \mbox{cof}(u_{\tau i j}) d\tau \right) \tu_{ij},
\end{equation*}
where $u_{\tau}(x,t)= \tau \underline{u}(x,t)+ (1-\tau)u(x,t) $ is strictly convex 
and the cofactor matrix $\mbox{cof}(u_{\tau i j})$ is positive definite on any compact subset of $\Omega \times (0,T]$ for any $T<\infty$. 
Next we choose $\underline{\vp}(0)$ and $\bar{\vp}(0)$ 
so that $\underline{\vp}(0)\psi(x) \leq u(x,0) \leq \bar{\vp}(0)\psi(x) $ on $\Omega$. 
Then  
\begin{equation}
\tu(x,0) \leq 0 \text{  in  } \ol{\Omega} \text{  and  } \tu(x,t) = 0  \text{  in  }\p \Omega \times [0, \infty),
\end{equation}
and we can then apply the classical maximum principle to conclude that 
$\tu(x,t)=\underline{u}(x,t)-u(x,t)\leq 0$ on $\ol{\Omega}\times [0,\infty)$. 
Consequently, 
\begin{equation}\label{4.1}
\left\{(1+t)^{\frac{1}{n \a -1}}\left(G (\ul{\vp}(0)^{1-n\a}+t)\right)^{\frac{1}{1-n\a}}-1 \right\}\psi(x)
\leq (1+t)^{\frac{1}{n \a -1}}u(x,t) -\psi(x). 
\end{equation}
Similarly, one derives that $u(x,t) \leq \bar{u}(x,t)$, namely, 
\begin{equation} \label{4.5}
(1+t)^{\frac{1}{n \a -1}}u(x,t) -\psi(x) \leq \left\{(1+t)^{\frac{1}{n \a -1}}(\ol{\vp}(0)^{1-n\a}+t)^{\frac{1}{1-n\a}}-1 \right\}\psi(x)
\end{equation}
Without loss of generality we may assume $\ul {\vp}(0) \geq 1$ and $\bar{\vp}(0) \leq 1$. 
Thus by Lemma \ref{Lemma F}, 
 $$F(\ul {\vp}(0)^{1-n\a},t) \leq 1+C_2/(1+t)$$ 
 $$F(\ol {\vp}(0)^{1-n\a},t) \geq 1-C_3/(1+t),$$ 
 where $C_2, C_3$ depend on $n, \a$ and $u_0(x)$. 
Combining now \eqref{4.1} and \eqref{4.5}, we arrive at that for all $t\geq 0$ and $x\in \ol{\Omega}$, 
\begin{equation}\label{4.10}
\left[\frac{C_2}{1+t} +G^{\frac{1}{1-n \a }}-1\right] \psi 
\leq  (1+t)^{\frac{1}{n \a -1}}u(x,t)-\psi 
\leq \frac{-C_3 \psi}{1+t},
\end{equation}
If $\b=0$, then $G=1$ and \eqref{4.10} implies \eqref{2.1} 
with $C_1=\max\{C_2,C_3\} \sup_{\Omega} |\psi|$. 
If $\b >0$, one needs to estimate $|\nabla u(x,t)|$ more carefully as V. Oliker did\cite[Pages 255-256]{O1}. 
Take an increasing sequence $t_m \to \infty$ and let $G_m = \inf_{\Omega} (1+|\nabla u(x,t_m)|^2)^{-\a \b}$. 
The same argument as in deriving \eqref{4.10} yields for all $t \geq t_m$ and $x\in \ol{\Omega}$, 
\begin{equation}\label{5.8}
\left[\frac{c_m}{1+t} +G_m^{\frac{-1}{n \a -1}}-1\right] \psi 
\leq (1+t)^{\frac{1}{n \a -1}}u(x,t)-\psi 
\leq \frac{-C_3 \psi}{1+t}.
\end{equation} 
where $c_m=\left(1-\underline{\vp}(t_m)^{1-n\a}\right) \underline{\vp}(t_m)^{n \a }/(n\a -1)< \infty$ 
uniformly in $m$ due to \eqref{4.1} . 
The same argument as in \cite{O1} allows one to 
let $t_m \to \infty$ and deduce \eqref{2.3}, hence completing the proof of Theorem \ref{Thm B}. 
\end{proof}
\begin{remark}
Similarly to \cite{AP81} one sees the sharpness of the estimate \eqref{5.8} by considering 
the function $u(x,t)=(s+t)^{\frac{1}{n \a -1}}\psi(x) $ for any $s>0$. 
\end{remark}
\begin{remark}
Corollary \ref{Cor 1} with $C_4 = \ul{\vp}(0)^{1-n\a}$ follows from $\underline{u}(x,t) \leq u(x,t)$, namely, 
$$G^{\frac{1}{1-n\a}} (\ul{\vp}(0)^{1-n\a}+t)^{\frac{1}{1-n\a}} \psi(x) \leq u(x,t). $$
\end{remark}



\bibliographystyle{alpha}

\bibliography{myref2015}
 
\end{document}